\documentclass[12pt,draftcls,onecolumn]{IEEEtran}

\usepackage{amsmath,mathrsfs}
\usepackage{amssymb}
\usepackage{pstricks,pst-plot,psfrag}
\usepackage[all]{xy}
\usepackage{graphicx,subfigure,xspace,bm}
\usepackage[lined,linesnumbered,algoruled,noend]{algorithm2e}
  
\newtheorem{theorem}{Theorem}[section]

\newtheorem{lemma}[theorem]{Lemma}

\newtheorem{remark}[theorem]{Remark}
\newtheorem{example}[theorem]{Example}

\newtheorem{proposition}[theorem]{Proposition}

\newcommand{\wdin}{\upsubscr{d}{in}{w}}

\newcommand{\wdout}{\upsubscr{d}{out}{w}}

\renewcommand{\SS}{\mathcal{S}_{\text{agree}}}
\newcommand{\WW}{\mathcal{W}}

\newcommand{\realpart}{\mathrm{Re}}
\newcommand{\impart}{\mathrm{Im}}

\newcommand{\real}{{\mathbb{R}}}
\newcommand{\realpositive}{{\mathbb{R}}_{>0}}
\newcommand{\realnonnegative}{{\mathbb{R}}_{\ge 0}}

\newcommand{\until}[1]{\{1,\dots,#1\}}
\newcommand{\map}[3]{#1:#2 \rightarrow #3}
\newcommand{\setmap}[3]{#1:#2 \rightrightarrows #3}
\newcommand{\equilibria}[1]{\operatorname{Eq}(#1)}

\newcommand{\Alin}{A_{\mathrm{linear}}}
\renewcommand{\Alin}{A}
\newcommand{\Bgraph}{\mathcal{G}}

\newcommand{\Lie}{\mathcal{L}}
\newcommand{\SetLie}{\widetilde{\mathcal{L}}}
\newcommand{\gradient}{\nabla}
\newcommand{\Bconvexhull}[1]{\mathrm{co}\{ #1\}}
\newcommand{\ones}{\mathbf{1}}
\newcommand{\zeros}{\mathbf{0}}
\newcommand{\setdef}[2]{\{#1 \; | \; #2\}}
\newcommand{\TwoNorm}[1]{\|#1\|_2}

\newcommand{\Lyap}{V}

\newcommand{\optdyn}{\subscr{\Psi}{dis-opt}}
\newcommand{\modoptdyn}{\subscr{\Psi}{$\alpha$-dis-opt}}

\newcommand{\Blap}{\mathbf{L}}

\newcommand{\fLift}{\tilde{f}}

\newcommand{\gLift}{\tilde{g}}
\newcommand{\xnet}{\bm{x}}
\newcommand{\ynet}{\bm{y}}
\newcommand{\znet}{\bm{z}}

\newcommand{\identity}[1]{\mathsf{I}_{#1}}
\newcommand{\setX}{\mathsf{X}}

\DeclareMathAlphabet{\mathpzc}{OT1}{pzc}{m}{it}

\newcommand{\partialvector}{\partial}
\newcommand{\Dout}{\subscr{\mathsf{D}}{out}}

\newcommand{\Adj}{\mathsf{A}}
\newcommand{\Lap}{\mathsf{L}}
\newcommand{\vertices}{\mathcal{V}}
\newcommand{\edges}{\mathcal{E}}

\newcommand{\smallnonzeroeig}{\Lambda_{*}}

\newcommand\subscr[2]{#1_{\textup{#2}}}

\newcommand\upsubscr[3]{#1_{\textup{#2}}^{\textup{#3}}}

\newcommand{\oprocendsymbol}{\hbox{$\bullet$}}
\newcommand{\oprocend}{\relax\ifmmode\else\unskip\hfill\fi\oprocendsymbol}

\newcommand{\longthmtitle}[1]{\mbox{}\textup{\textsl{(#1):}}}

\renewcommand{\theenumi}{(\roman{enumi})}

\renewcommand{\baselinestretch}{1.41}

\newcommand{\myclearpage}{\clearpage}

\renewcommand{\myclearpage}{}

\begin{document}

\title{Distributed continuous-time convex optimization on
  weight-balanced digraphs}
  
\author{Bahman Gharesifard \qquad Jorge Cort\'{e}s\thanks{Bahman
    Gharesifard and Jorge~Cort\'{e}s are with the Department of
    Mechanical and Aerospace Engineering, University of California,
    San Diego, \texttt{\{bgharesifard,cortes\}@ucsd.edu}.}}

\maketitle

\begin{abstract}
  This paper studies the continuous-time distributed optimization of a
  sum of convex functions over directed graphs.  Contrary to what is
  known in the consensus literature, where the same dynamics works for
  both undirected and directed scenarios, we show that the
  consensus-based dynamics that solves the continuous-time distributed
  optimization problem for undirected graphs fails to converge when
  transcribed to the directed setting.  This study sets the basis for
  the design of an alternative distributed dynamics which we show is
  guaranteed to converge, on any strongly connected weight-balanced
  digraph, to the set of minimizers of a sum of convex differentiable
  functions with globally Lipschitz gradients.  Our technical approach
  combines notions of invariance and cocoercivity with the positive
  definiteness properties of graph matrices to establish the results.
\end{abstract}

\vspace*{-1ex}
\section{Introduction}\label{section:intro}

Distributed optimization of a sum of convex functions has applications
in a variety of scenarios, including sensor networks, source
localization, and robust estimation, and has been intensively studied
in recent years, see
e.g.~\cite{MR-RN:04,AN-AO:09,PW-MDL:09,AN-AO-PAP:10,BJ-MR-MJ:09,MZ-SM:12,JNT-DPB-MA:86}.
Most of these works build on consensus-based
dynamics~\cite{ROS-JAF-RMM:07,WR-RWB:08,FB-JC-SM:08cor,MM-ME:10} to
design discrete-time algorithms that find the solution of the
optimization problem.
A recent exception are the works~\cite{JW-NE:10, JW-NE:11} that deal
with continuous-time strategies on undirected networks. This paper
furthers contributes to this body of work by studying continuous-time
algorithms for distributed optimization in directed scenarios.

The unidirectional information flow among agents characteristic of
directed networks often leads to significant technical challenges when
establishing convergence and robustness properties of coordination
algorithms.  The results of this paper provide one more example in
support of this assertion for the case of continuous-time
consensus-based distributed optimization.  This is somewhat surprising
given that, for consensus, the same dynamics works for both undirected
connected graphs and strongly connected, weight-balanced directed
graphs, see e.g.,~\cite{ROS-JAF-RMM:07,WR-RWB:08}.


The contributions of this paper are the following.  We first show that
the solutions of the optimization problem of a sum of locally
Lipschitz convex functions over a directed graph (or digraph)
correspond to the saddle points of an aggregate objective function
that depends on the graph topology through its Laplacian. This
function is convex in its first argument and linear in the
second. Moreover, its gradient is distributed when the graph is
undirected. Our second step is then to study the convergence
properties of the saddle-point dynamics and establish its asymptotic
correctness when the original functions are locally Lipschitz (i.e.,
not necessarily differentiable) and convex, extending the results
available in the literature~\cite{JW-NE:11} for continuously
differentiable, strictly convex functions.  Next, we consider the
optimization problem over digraphs. We first provide an example of a
strongly connected, weight-balanced digraph where the distributed
version of the saddle-point dynamics does not converge. This motivates
us to introduce a generalization of the dynamics that incorporates a
design parameter. We show that, when the original functions are
differentiable and convex with globally Lipschitz gradients, the
design parameter can be appropriately chosen so that the resulting
dynamics asymptotically converge to the set of minimizers of the
objective function on any strongly connected and weight-balanced
digraph. Our technical approach combines notions and tools from
set-valued stability analysis, algebraic graph theory, and convex
analysis. 

\myclearpage%
\vspace*{-1ex}
\section{Preliminaries}\label{section:prelim}

We start with notational conventions.  Let $\real$ and
$\realnonnegative$ denote the set of reals and nonnegative reals,
respectively.  We let $ ||\cdot || $ denote the Euclidean norm on $
\real^d$.
We let $ \ones_d=(1,\ldots,1)^T $, $\zeros_d=(0,\ldots,0)^T \in
\mathbb{R}^d$, and $ \identity{d} $ denote the identity matrix in $
\mathbb{R}^{d\times d}$.  For $ A\in \real^{d_1\times d_2} $ and $ B
\in \real^{e_1\times e_2} $,
$ A\otimes B $ is their Kronecker product.  A function
$\map{f}{\setX_1\times \setX_2}{\mathbb{R}} $, with $ \setX_1 \subset
\real^{d_1} $, $ \setX_2\subset \real^{d_2} $ closed and convex, is
\emph{concave-convex} if it is concave in its first argument and
convex in the second one.  A \emph{saddle point} $ (x_1^*,x_2^*) \in
\setX_1\times \setX_2 $ of $ f $ satisfies $ f(x_1,x_2^*)\leq
f(x^*_1,x^*_2)\leq f(x_1^*,x_2)$ for all $ x_1 \in \setX_1 $ and $ x_2
\in \setX_2$. A set-valued map $\setmap{f}{\real^d}{\real^d}$ takes
elements of $\real^d$ to subsets of~$\real^d$.

\vspace*{-1ex}
\subsection{Graph theory}

We present basic notions from algebraic graph
theory~\cite{FB-JC-SM:08cor}.  A \emph{directed graph}, or
\emph{digraph}, is a pair $\Bgraph=(\vertices,\edges)$, where
$\vertices$ is the (finite) vertex set and $ \edges \subseteq
\vertices\times \vertices $ is the edge set.
A digraph is \emph{undirected} if $(v,u) \in \edges$ anytime $(u,v) \in
\edges$. We refer to an undirected digraph as a \emph{graph}.
A path is an ordered sequence of vertices such that any pair of
vertices appearing consecutively is an edge.  A digraph is
\emph{strongly connected} if there is a path between any pair of
distinct vertices. For a graph, this notion is referred to as
\emph{connected}.  A \emph{weighted digraph} is a triplet $
\Bgraph=(\vertices,\edges,\Adj) $, where $ (\vertices,\edges) $ is a
digraph and $ \Adj \in \mathbb{R}^{n\times n}_{\geq0} $ is the
\emph{adjacency matrix}, satisfying $ a_{ij}>0 $ if $ (v_i,v_j)\in
\edges $ and $ a_{ij}=0 $, otherwise.  The weighted out-degree and
in-degree of $v_i$, $i \in \{1,\dots,n\}$, are respectively, $
\wdout(v_i) =\sum_{j=1}^{n}a_{ij} $ and $\wdin(v_i)=\sum_{j=1}^n
a_{ji} $.  The \emph{weighted out-degree matrix} $ \Dout$ is diagonal
with $ (\Dout)_{ii}=\wdout(i) $, for $ i \in \{1,\ldots,n\}$.  The
\emph{Laplacian} matrix is $ \Lap = \Dout -\Adj$. Note that $
\Lap\ones_n=0 $.  If $ \Bgraph $ is strongly connected, then zero is a
simple eigenvalue of $\Lap$.  $\Bgraph$ is undirected if $ \Lap=\Lap^T
$ and \emph{weight-balanced} if $ \wdout(v) =\wdin(v) $, for all $ v
\in \vertices$. The following three notions are equivalent: (i)
$\Bgraph $ is weight-balanced, (ii) $ \ones_n^T \Lap =0 $, and (iii) $
\Lap+\Lap^T $ is positive semidefinite, see
e.g.,~\cite[Theorem~1.37]{FB-JC-SM:08cor}.  If $\Bgraph $ is
weight-balanced and strongly connected, then zero is a simple
eigenvalue of $ \Lap+\Lap^T $.  Any undirected graph is
weight-balanced.

\vspace*{-1ex}
\subsection{Nonsmooth analysis}

We recall some notions from nonsmooth analysis~\cite{FHC:83}.  A
function $ \map{f}{\real^d}{\real} $ is \emph{locally Lipschitz} at $
x \in \real^d $ if there exists a neighborhood $ \mathcal{U} $ of $ x
$ and $ C_x \in \realnonnegative$ such that $ |f(y)-f(z)|\leq C_x
||y-z|| $, for $ y,z \in \mathcal{U} $. $f$ is locally Lipschitz on
$\real^d$ if it is locally Lipschitz at $x$ for all $x \in \real^d$
and \emph{globally Lipschitz} on $ \real^d $ if for all $ y,z \in
\real^d $ there exists $ C \in \realnonnegative $ such that $
|f(y)-f(z)|\leq C ||y-z|| $.  Locally Lipschitz functions are
differentiable almost everywhere.  If $ \Omega_f $ denotes the set of
points where $ f $ fails to be differentiable, the \emph{generalized
  gradient} of $ f $ is
\[
\partial f(x) = \Bconvexhull{\lim_{k \rightarrow \infty} \nabla f(x_k)
  \ | \ x_k \rightarrow x, x_k \notin \Omega_f \cup S},
\]
where $ S $ is any set of measure zero and $\text{co} $ denotes convex
hull.

\begin{lemma}\longthmtitle{Continuity of the generalized gradient
    map}\label{le:gradient_properties}
  Let $ \map{f}{\real^d}{\real} $ be a locally Lipschitz function at $
  x \in \real^d $.  Then the set-valued map $ \setmap{\partial
    f}{\real^d}{\real^d} $ is upper semicontinuous and locally bounded
  at $ x \in \real^d $ and moreover, $ \partial f(x) $ is nonempty,
  compact, and convex.
\end{lemma}

For $\map{f}{\real^d \times \real^d}{\real}$ and $z \in \real^d$, we
let $ \partial_x f(x,z) $ denote the generalized gradient of $x
\mapsto f(x,z)$. Similarly, for $x \in \real^d$, we let $ \partial_z
f(x,z) $ denote the generalized gradient of $z \mapsto f(x,z)$.  A
\emph{critical point} $ x \in \real^d $ of $ f $ satisfies $ \zeros
\in \partial f(x) $.  A function $ \map{f}{\real^d}{\real} $ is
\emph{regular} at $ x \in \real $ if for all $ v \in \real^d $ the
right directional derivative of $ f $, in the direction of $ v $,
exists at $ x $ and coincides with the generalized directional
derivative of $ f $ at $ x$ in the direction of~$v$, see~\cite{FHC:83}
for definitions of these notions. A convex and locally Lipschitz
function at $ x $ is regular~\cite[Proposition~2.3.6]{FHC:83}.

\begin{lemma}\longthmtitle{Finite sum of
    locally Lipschitz functions}\label{le:finite_sum_gen}
  Let $ \{f^i\}_{i=1}^n$ be locally Lipschitz at $ x \in \real^d
  $. Then $\partial(\sum_{i=1}^nf^i)(x)\subseteq \sum_{i=1}^n\partial
  f^i(x)$, and equality holds if $ f^i $ is regular for $ i \in
  \{1,\ldots, n\} $ (here, the summation of sets is the set of points
  of the form $ \sum_{i=1}^ng_i $, with $ g_i \in \partial f^i(x) $).
\end{lemma}
A locally Lipschitz and convex function $f$ satisfies, for all $x,x'
\in \real^d$ and $ \xi \in \partial f(x) $, the \emph{first-order
  condition} of convexity,
\begin{equation}\label{eq:convex_prop}
  f(x')-f(x) \geq \xi^T (x'-x).
\end{equation}
The notion of cocoercivity~\cite{EGG-NVY:96} plays a key role in our
technical approach later.  For $ \delta \in \realpositive $, a locally
Lipschitz function $f$ is $ \delta$-\emph{cocoercive} if, for all $ x,
x' \in \real^d $ and $ g_x \in \partial f(x) $, $ g_{x'} \in \partial
f(x') $,
\begin{align*}
  (x-x')^T(g_x-g_{x'})\geq \delta (g_x-g_{x'})^T(g_x-g_{x'}).
\end{align*}
The next result~\cite[Lemma~6.7]{EGG-NVY:96} characterizes cocoercive
differentiable convex functions.

\begin{proposition}\longthmtitle{Characterization of 
    cocoercivity}\label{prop:coco_nes_suff}
  Let $ f $ be a differentiable convex function. Then, $\nabla f$ is
  globally Lipschitz with constant $ K \in \realpositive $ iff $ f $
  is $ \frac{1}{K}$-cocoercive.
\end{proposition}

\vspace*{-1ex}
\subsection{Set-valued dynamical systems}

Here, we recall some background on set-valued dynamical systems
following~\cite{
  JC:08-csm-yo}.  A continuous-time set-valued dynamical system on
$\setX \subset \mathbb{R}^d$ is a differential inclusion
\begin{equation}\label{eq:diff-inc}
  \dot{x}(t) \in \Psi(x(t))
\end{equation}
where $ t \in \realnonnegative $ and $\setmap{\Psi}{\setX \subset
  \real^d}{\real^d}$ is a set-valued map.  A solution to this
dynamical system is an absolutely continuous curve $ x:
[0,T]\rightarrow \setX $ which satisfies~\eqref{eq:diff-inc} almost
everywhere. The set of equilibria of~\eqref{eq:diff-inc} is denoted by
$\equilibria{\Psi} = \setdef{x \in \setX}{ 0 \in \Psi(x)}$.  

\begin{lemma}\longthmtitle{Existence of solutions}\label{le:solution}
  For $\setmap{\Psi}{\real^d}{\real^d}$ upper semicontinuous with
  nonempty, compact, and convex values, there exists a solution
  to~\eqref{eq:diff-inc} from any initial condition.
\end{lemma}

The LaSalle Invariance Principle 
is helpful to establish the asymptotic convergence of systems
of the form~\eqref{eq:diff-inc}.  A set $ W \subset \setX $ is
\emph{weakly positively invariant} under~\eqref{eq:diff-inc} if, for
each $x \in W$, there exists at least one solution
of~\eqref{eq:diff-inc} starting from $x$ entirely contained in $W$.
Similarly, $ W $ is \emph{strongly positively invariant}
under~\eqref{eq:diff-inc} if, for each $x \in W$, all solutions
of~\eqref{eq:diff-inc} starting from $x$ are entirely contained in
$W$.  Finally, the \emph{set-valued Lie derivative} of a
differentiable function $ \map{V}{\real^d}{\real} $ with respect to $
\Psi $ at $ x \in \real^d $ is $ \SetLie_{\Psi}{V(x)}=\{v^T \nabla
V(x) \ | \ v \in \Psi(x) \}$.


\begin{theorem}\longthmtitle{Set-valued LaSalle Invariance
    Principle}\label{th:laSalle}
  Let $ W \subset \setX $ be strongly positively invariant
  under~\eqref{eq:diff-inc} and $ \Lyap: \setX \rightarrow \mathbb{R}
  $ a continuously differentiable function.  Suppose the evolutions
  of~\eqref{eq:diff-inc} are bounded and $ \max
  \SetLie_{\Psi}{\Lyap(x)} \leq 0 $ or $
  \SetLie_{\Psi}{\Lyap(x)}=\emptyset $, for all $ x \in W $. Let
  $ S_{\Psi,\Lyap} = \{x\in \setX \ | \ 0 \in
  \SetLie_{\Psi}{\Lyap(x)}\} $.
  Then any solution $ x(t) $, $ t\in \mathbb{R}_{\geq 0} $, starting
  in $W$ converges to the largest weakly positively invariant set $ M
  $ contained in $ \bar{S}_{\Psi,\Lyap}\cap W$. When $ M $ is a finite
  collection of points, then the limit of each solution 
  equals one of them.
\end{theorem}

\myclearpage

\vspace*{-1ex}
\section{Problem statement and equivalent
  formulations}\label{section:problem-statement}

Consider a network composed by $n$ agents $v_1,\dots, v_n $ whose
communication topology is described by a strongly connected digraph
$\Bgraph$. An edge $(v_i,v_j)$ represents the fact that $v_i$ can
receive information from $v_j$.  For each $i \in \{1,\ldots, n\} $,
let $ \map{f^i}{\mathbb{R}^d}{\mathbb{R}} $ be locally Lipschitz and
convex, and only available to agent $v_i$. The network objective is to
solve
\begin{align}\label{eq:dis_opt}
  \mathrm{minimize} \quad f(x)=\sum_{i=1}^n f^i(x) ,
\end{align}
in a distributed way.  Let $ x^i\in \mathbb{R}^d $ denote the estimate
of agent $ v_i $ about the value of the solution to~\eqref{eq:dis_opt}
and let $\xnet^T = ((x^1)^T, \dots, (x^n)^T) \in \real^{n d}$. Next,
we provide an alternative formulation of~\eqref{eq:dis_opt}.

\begin{lemma}\label{le:equiv-ftilde}
  Let $ \Lap \in \real^{n \times n} $ be the Laplacian of $\Bgraph$
  and define $ \Blap= \Lap \otimes \identity{d} \in \real^{nd \times
    nd}$.  The problem~\eqref{eq:dis_opt} on $\real^d$ is equivalent
  to the following problem on $\real^{nd}$,
  \begin{align}\label{eq:dis_opt_2}
    \mathrm{minimize} \quad \fLift(\xnet) = \sum_{i=1}^nf^i(x^i),
    \qquad \text{subject to} \quad \Blap \xnet = \zeros_{nd} .
  \end{align}
\end{lemma}
\begin{proof}
  The proof follows by noting that (i) $\fLift (\ones_n \otimes x) =
  f(x)$ for all $x \in \real^d$ and (ii) since $\Bgraph$ is strongly
  connected, $\Blap \xnet = \zeros_{nd}$ if and only if $\xnet =
  \ones_n \otimes x$, for some $x \in \real^d$.
\end{proof}

The formulation~\eqref{eq:dis_opt_2} is appealing because it brings
together the estimates of each agent about the value of the solution
to the original optimization problem.  Note that $\fLift$ is locally
Lipschitz and convex. Moreover, from Lemma~\ref{le:finite_sum_gen},
the elements of its generalized gradient are of the form $
\gLift_{\xnet} = (g^1_{x^1}, \ldots, g^n_{x^n}) \in \partialvector
\fLift(\xnet), $ where $ g^i_{x^i} \in \partial f^i(x^i) $, for $ i
\in \{1,\ldots, n\}$.  Since $ \fLift $ is convex and the constraints
in~\eqref{eq:dis_opt_2} are linear, the constrained optimization
problem is feasible~\cite{SB-LV:04}.

The next result introduces a function which corresponds to the
Lagrangian function associated to the constrained optimization
problem~\eqref{eq:dis_opt_2} plus an additional quadratic term that
vanishes if the agreement constraint is satisfied. Interestingly, the
saddle points of this function correspond to the solutions of the
constrained optimization problem, as we show next.

\begin{proposition}\longthmtitle{Solutions of the distributed
    optimization problem as saddle points}\label{prop:equiv-F}
  Let $\Bgraph$ be strongly connected and weight-balanced, and define
  $ \map{F}{\real^{nd} \times \real^{nd} }{\real} $ by
  \begin{equation}\label{eq:F_opt}
    F(\xnet,\znet) = \fLift(\xnet)+\xnet^T\Blap\znet +
    \frac{1}{2}\xnet^T\Blap\xnet.
  \end{equation}
  Then $ F $ is locally Lipschitz and convex in its first argument and
  linear in its second, and 
  \begin{enumerate}
  \item if $ (\xnet^*,\znet^*)$ is a saddle point of $F$, then so is $
    (\xnet^*,\znet^*+\ones_{n}\otimes a)$, for any $ a \in \real^d $.
  \item if $ (\xnet^*,\znet^*) $ is a saddle point of $F$, then $
    \xnet^* $ is a solution of~\eqref{eq:dis_opt_2}.
  \item if $ \xnet^* $ is a solution of~\eqref{eq:dis_opt_2}, then
    there exists $\znet^*$ with $\Blap \znet^* \in - \partialvector
    \fLift(\xnet^*)$ such that $ (\xnet^*,\znet^*) $ is a saddle point
    of $F$.
  \end{enumerate}
\end{proposition}
\begin{proof}
  First, note that for $\Bgraph $ weight-balanced, $\Blap + \Blap^T$
  is positive semi-definite. Since the sum of convex functions is
  convex, one deduces that $ F $ is convex in its first argument. By
  inspection, $F$ is linear in its second argument.  The statement~(i)
  is immediate. To show~(ii), using that $\Bgraph$ is strongly
  connected, one can see that the saddle points of~$F$ are of the form
  $(\xnet^*,\znet^*) $ with $\xnet^* = \ones_{n}\otimes x^* $, $ x^*
  \in \real^{d} $, and $ \Blap\znet^* \in -\partial \fLift(\xnet^*)$.
  The last inclusion implies that there exist $g^i_{x^*} \in \partial
  f^i(x^*) $, $ i \in \{1,\ldots, n\} $, such that $\Blap \znet^* = -
  (g^1_{x^*}, \ldots, g^n_{x^*})^T$.  Noting that
  \begin{align*}
    (\ones_{n}^T\otimes \identity{d}) \Blap = (\ones_{n}^T\otimes
    \identity{d}) (\Lap \otimes \identity{d}) = \ones_{n}^T \Lap
    \otimes \identity{d} = \zeros_{d \times dn} ,
  \end{align*}
  we deduce $ \mathbf{0}_d = (\ones_{n}^T\otimes \identity{d}) \Blap
  \znet^* =- \sum_{i=1}^n g^i_{x^*} $.  As a result, using
  Lemma~\ref{le:finite_sum_gen}, $ \xnet^* $ is a solution
  of~\eqref{eq:dis_opt_2}.  Finally, (iii) follows by noting $\xnet^*
  = \ones_{n}\otimes x^*$ and the fact that $0 \in \partial f (x^*)$
  implies that there exists $\znet^*\in \real^{nd} $ with $\Blap
  \znet^* \in - \partialvector \fLift(\xnet^*)$, yielding that $
  (\xnet^*,\znet^*) $ is a saddle point of~$F$.
\end{proof}


\myclearpage
\vspace*{-1ex}
\section{Continuous-time distributed optimization on
undirected networks}\label{section:dis_opt_undirected}

Here, we review the continuous-time solution to the optimization
problem proposed in~\cite{JW-NE:10,JW-NE:11} for undirected graphs.
If $\Bgraph$ is undirected, the gradient of $F$ in~\eqref{eq:F_opt} is
distributed over $\Bgraph$. Given Proposition~\ref{prop:equiv-F}, it
is natural to consider the saddle-point dynamics of~$F$ to
solve~\eqref{eq:dis_opt},
\begin{subequations}\label{eq:CT-Laplacian-optimization}
  \begin{align}
    \dot \xnet + \Blap\xnet + \Blap\znet & \in -\partial
    \fLift(\xnet),
    \\
    \dot \znet & = \Blap \xnet.
  \end{align}
\end{subequations}
Note that~\eqref{eq:CT-Laplacian-optimization} is a set-valued
dynamical system.
Using Lemmas~\ref{le:gradient_properties} and~\ref{le:solution}, one
can guarantee the existence of solutions. Moreover, from
Proposition~\ref{prop:equiv-F}, if $(\xnet^*,\znet^*)$ is an
equilibrium of~\eqref{eq:CT-Laplacian-optimization}, then $ \xnet^* $
is a solution to~\eqref{eq:dis_opt_2}.  According to~\cite{JW-NE:11},
the dynamics~\eqref{eq:CT-Laplacian-optimization} leads the network to
agree on a global minimum of~$f$ for the case when~$\Bgraph$ is
undirected and~$f$ is both strictly convex and the sum of
differentiable convex functions.  We extend here this
result to the case when $\Bgraph$ is undirected and~$f$ is the sum of
locally Lipschitz convex functions. The proof is also useful later to
illustrate the challenges in solving the distributed optimization
problem over directed graphs.

\begin{theorem}\longthmtitle{Asymptotic convergence
    of~\eqref{eq:CT-Laplacian-optimization} on graphs}\label{theorem:dis_opt}
  Let $\Bgraph$ be a connected graph and consider the optimization
  problem~\eqref{eq:dis_opt}, where each $ f^i $, $i \in \until{n}$ is
  locally Lipschitz and convex. Then, the projection onto the first
  component of any trajectory of~\eqref{eq:CT-Laplacian-optimization}
  asymptotically converges to the set of solutions
  to~\eqref{eq:dis_opt_2}. Moreover, if $ f $ has a finite number of
  critical points, the limit of the projection onto the first
  component of each trajectory 
  is a solution of~\eqref{eq:dis_opt_2}.
\end{theorem}

\begin{proof}
  For convenience, we denote the
  dynamics~\eqref{eq:CT-Laplacian-optimization} by
  $\setmap{\optdyn}{\real^{n d} \times \real^{n d}}{\real^{n d} \times
    \real^{n d}}$.  Let $\xnet^* = \ones_n \otimes x^*$ be a solution
  of~\eqref{eq:dis_opt_2}. By Proposition~\ref{prop:equiv-F}(iii),
  there exists $\znet^*$ such that $ (\xnet^*,\znet^*) \in
  \equilibria{\optdyn} $.  First, note that given any initial
  condition $ (\xnet_0,\znet_0) \in \real^{nd} \times \real^{nd} $,
  the set 
  \begin{equation}\label{eq:S_z0}
    W_{\znet_0} = \{(\xnet,\znet) \ | \ (\ones_{n}^T \otimes
    \identity{d}) \znet = (\ones_{n}^T \otimes \identity{d}) \znet_0\}
  \end{equation}
  is strongly positively invariant
  under~\eqref{eq:CT-Laplacian-optimization}. Consider then the
  function $ \map{V}{\real^{nd} \times \real^{nd}}{\real_{\geq0}} $,
  \begin{equation}\label{eq:Lyap}
    \Lyap(\xnet,\znet) = \frac{1}{2}(\xnet-\xnet^*)^T(\xnet-\xnet^*) +
    \frac{1}{2}(\znet-\znet^*)^T(\znet-\znet^*).
  \end{equation}
  The function $V$ is smooth. Let us examine its set-valued Lie
  derivative. For each $ \xi \in \SetLie_{\optdyn}V(\xnet,\znet) $,
  there exists $v = (-\Blap\xnet-\Blap \znet -\gLift_{\xnet},
  \Blap\xnet) \in \optdyn(\xnet,\znet) $, with $ \gLift_{\xnet}
  \in \partial \fLift(\xnet) $, such that
  \begin{align}\label{eq:xi-undirected}
    \xi = v^T \nabla V(\xnet,\znet) = -(\xnet-\xnet^*)^T (
    \Blap\xnet+\Blap \znet + \gLift_{\xnet}) + (\znet-\znet^*)^T \Blap
    \xnet .
  \end{align}
  Since $F$ is convex in its first argument and $\Blap\xnet+\Blap
  \znet + \gLift_{\xnet} \in \partial_{\xnet}F(\xnet,\znet)$, using
  the first-order condition of convexity~\eqref{eq:convex_prop}, we
  deduce $ (\xnet^*-\xnet)^T( \Blap\xnet+\Blap \znet + \gLift_{\xnet})
  \le F(\xnet^*, \znet)-F(\xnet, \znet)$.  On the other hand, the
  linearity of $F$ in its second argument implies that $
  (\znet-\znet^*)^T \Blap \xnet = F(\xnet, \znet)-F(\xnet, \znet^*)$.
  Therefore,
  $\xi \leq F(\xnet^*, \znet) - F(\xnet^*, \znet^*) + F(\xnet^*,
  \znet^*)- F(\xnet, \znet^*) $.
  Since the equilibria of $\optdyn$ are the saddle points of $F$, we
  deduce that $ \xi \leq 0$.  Since $ \xi $ is arbitrary, we conclude
  $ \max \SetLie_{\optdyn}V (\xnet,\znet) \leq 0$. As a by-product,
  the trajectories of~\eqref{eq:CT-Laplacian-optimization} are
  bounded. Consequently, all assumptions of the set-valued version of
  the LaSalle Invariance Principle, cf. Theorem~\ref{th:laSalle}, are
  satisfied. This result then implies that any trajectory
  of~\eqref{eq:CT-Laplacian-optimization} starting from an initial
  condition $ (\xnet_0,\znet_0) $ converges to the largest weakly
  positively invariant set $ M $ in $ S_{\optdyn,\Lyap} \cap
  W_{\znet_0}$. Our final step consists of characterizing $ M $.  Let
  $ (\xnet,\znet) \in M$.  Then $F(\xnet^*, \znet^*) - F(\xnet,
  \znet^*) = 0 $, i.e.,
  \begin{align}\label{eq:aux-eq}
    \fLift(\xnet^*) -\fLift(\xnet)-(\znet^*)^T\Blap \xnet
    -\frac{1}{2}\xnet^T\Blap \xnet =0 .
  \end{align}
  Define now $\map{G}{\real^{nd} \times \real^{nd} }{\real} $ by
  $G(\xnet,\znet) = \fLift(\xnet)+\znet^T\Blap\xnet$. Note that $G$ is
  convex in its first argument and linear in its second, and that it
  has the same saddle points as $ F $. As a result, $
  G(\xnet^*,\znet^*)-G(\xnet,\znet^*) \leq 0$, or equivalently,
  $\fLift(\xnet^*) -\fLift(\xnet)-(\znet^*)^T \Blap \xnet\leq 0 $.
  Combining this with~\eqref{eq:aux-eq}, we have $ \Blap \xnet=0$ and
  $ -\fLift(\xnet)+\fLift(\xnet^*) = 0$, i.e., $ \xnet $ is solution
  to~\eqref{eq:dis_opt_2}.  Since $ M $ is weakly positively
  invariant, there exists at least a solution
  of~\eqref{eq:CT-Laplacian-optimization} starting from
  $(\xnet,\znet)$ that remains in $M$. This implies that, along the
  solution, the components of $\xnet$ remain in agreement, i.e., $
  \xnet(t) = \ones_{n} \otimes a(t)$ with $ a(t) \in \real^d $ a
  solution of~\eqref{eq:dis_opt}.  Applying $ \ones_n^T\otimes
  \identity{d} $ on both sides of $ \ones_{n} \otimes \dot{a}(t)
  +\Blap \znet \in -\partial \fLift(\xnet(t)) $, we deduce $
  n\dot{a}(t) \in -\sum_{i=1}^n\partial f^i( a(t)) $.
  Lemma~\ref{lemma:flow-in-critical} then implies that $ \dot{a}(t)=0
  $, i.e., $ \Blap \znet \in -\partial \fLift(\xnet) $ and thus $
  (\xnet,\znet)\in \equilibria{\optdyn} $.  Finally, if the set of
  equilibria is finite, the last statement holds true.
\end{proof}

\begin{remark}\longthmtitle{Asymptotic convergence of saddle-point
    dynamics}
  The work~\cite{KA-LH-HU:58} studies saddle-point dynamics and
  guarantees asymptotic convergence to a saddle point when the
  function's Hessian in one argument is positive definite and the
  function is linear in the other.  Such result, however, cannot be
  applied to establish Theorem~\ref{theorem:dis_opt} because the
  generality of the hypotheses on~$f$ mean that~$F$ might not satisfy
  these conditions. Instead, our proof shows that a careful study of
  the invariance properties of the flow yields the desired result.
  \oprocend
\end{remark}

\vspace*{-.5ex}
\section{Continuous-time distributed optimization on directed
  networks}\label{section:dis_opt_directed}


Here, we consider the optimization problem~\eqref{eq:dis_opt} on
digraphs.  When $\Bgraph$ is directed, the gradient of $F$ defined
in~\eqref{eq:F_opt} is no longer distributed over $\Bgraph$ because it
contains terms that involve $\Blap^T$ and hence requires agents to
receive information from its in-neighbors. In fact, the
dynamics~\eqref{eq:CT-Laplacian-optimization}, which is distributed
over $\Bgraph$, does no longer correspond to the saddle-point dynamics
of $F$. Nevertheless, it is natural to study
whether~\eqref{eq:CT-Laplacian-optimization} enjoys the same
convergence properties as in the undirected setting (as, for instance,
is the case in the agreement
problem~\cite{ROS-JAF-RMM:07,WR-RWB:08}).  Surprisingly, this
turns out not to be the case, as shown in
Section~\ref{se:surprise}. This result motivates the introduction in
Section~\ref{se:the-right-one} of an alternative provably correct
dynamics on weight-balanced directed graphs.


\vspace{-2ex}
\subsection{Counterexample}\label{se:surprise}

Here, we provide an example of a strongly connected, weight-balanced
digraph on which~\eqref{eq:CT-Laplacian-optimization} fails to
converge.  For convenience, we let $\SS= \setdef{(\ones_{n} \otimes x,
  \ones_{n} \otimes z) \in \real^{nd} \times \real^{nd}}{x,z \in
  \real^d}$ denote the set of agreement configurations. Our
construction relies on the following result.

\begin{lemma}\longthmtitle{Necessary condition for the convergence
    of~\eqref{eq:CT-Laplacian-optimization} on
    digraphs}\label{lemma:necessary_linear_directed} 
  Let $\Bgraph$ be a strongly connected digraph and $ f^i=0$, $ i \in
  \{1,\ldots, n\} $. Then $\SS$ is stable
  under~\eqref{eq:CT-Laplacian-optimization} iff, for any nonzero
  eigenvalue $\lambda$ of the Laplacian $ \Lap $, one has
  $\sqrt{3}|\impart (\lambda)| \le \realpart (\lambda) $.
\end{lemma}
\begin{proof}
  By assumption, the dynamics~\eqref{eq:CT-Laplacian-optimization} is
  linear with matrix $\left(\begin{smallmatrix} -1 &
      -1\\1 & 0\end{smallmatrix} \right)\otimes  \Blap$ and has $\SS$ as
  equilibria. The eigenvalues of the matrix are of the form $\lambda
  \, \big(\frac{-1}{2}\pm\frac{\sqrt{3}}{2}i \big)$, with $\lambda$ an
  eigenvalue of $\Blap$ (because the eigenvalues of a Kronecker
  product are just the product of the eigenvalues of the corresponding
  matrices). Since $\Blap = \Lap \otimes \identity{d}$, each
  eigenvalue of $\Blap$ is an eigenvalue of $ \Lap $.  Finally, $
  \realpart \big(\lambda \big(\frac{-1}{2}\pm\frac{\sqrt{3}}{2}i
  \big)\big) = \frac{1}{2} ( \mp
  \sqrt{3}\impart(\lambda)-\realpart(\lambda))$, from which the result
  follows.
\end{proof}

It is not difficult to construct examples of convex functions that
have zero contribution to the linearization
of~\eqref{eq:CT-Laplacian-optimization} around the
solution. Therefore, such systems cannot be convergent if they fail
the criterium identified in
Lemma~\ref{lemma:necessary_linear_directed}. The next example shows
that this criterium can fail even for strongly connected
weight-balanced digraphs.

\begin{example}\label{example:counter}
  Consider the strongly connected, weight-balanced digraph
  with 
  \[
  A =
  \begin{pmatrix}
      0& 0.5326  &  0.1654 & 0.0004  & 0.0002\\
      0.0595   &      0  &  0.6676  &  0.0681  &  0.1230\\
      0.0213  &  0.0004 & 0  &  0.5809  &  0.3181\\
      0.0248  &  0.2458   &      0   &      0  &  0.5587\\
      0.5930  &  0.1394  &  0.0877  &  0.1799   &  0
    \end{pmatrix}
  \]
  as adjacency matrix.  Note that $ \lambda= 0.8833\pm 0.5197i $ is an
  eigenvalue of the Laplacian. Since $\sqrt{3}
  |\impart(\lambda)|-\realpart(\lambda)=0.0171>0$,
  Lemma~\ref{lemma:necessary_linear_directed} implies
  that~\eqref{eq:CT-Laplacian-optimization} fails to converge.
  \oprocend
\end{example}


\vspace*{-1ex}
\subsection{Provably correct distributed dynamics on directed
  graphs}\label{se:the-right-one}

Here, given the result in Section~\ref{se:surprise}, we introduce an
alternative continuous-time distributed dynamics for strongly
connected weight-balanced digraphs. For reasons that will be made
clear later in Remark~\ref{remark:diff_assumption}, we restrict our
attention to the case when the functions $f^i$, $i \in \until{n}$ are
continuously differentiable. 
Let $ \alpha \in \realpositive $ and consider the dynamics
\begin{subequations}\label{eq:CT-Laplacian-optimization-NEW}
  \begin{align}
    \dot \xnet + \alpha \Blap\xnet + \Blap\znet & = -\gradient
    \fLift(\xnet),
    \\
    \dot \znet & = \Blap \xnet.
  \end{align}  
\end{subequations}
The existence of solutions is guaranteed by
Lemmas~\ref{le:gradient_properties} and~\ref{le:solution}.  We first
show that appropriate choices of $\alpha$ allow to circumvent the
problem raised in Lemma~\ref{lemma:necessary_linear_directed}.

\begin{lemma}\longthmtitle{Sufficient conditions for the convergence
    of~\eqref{eq:CT-Laplacian-optimization-NEW} on digraphs with
    trivial objective 
    function}\label{lemma:necessary_linear_directed-NEW} 
  Let $\Bgraph$ be a strongly connected and weight-balanced digraph
  and $f^i=0$, $ i \in \{1,\ldots, n\} $. If $ \alpha \geq 2
  \sqrt{2}$, then $\SS$ is asymptotically stable
  under~\eqref{eq:CT-Laplacian-optimization-NEW}.
\end{lemma}

\begin{proof}
  When all $ f_i $, $i \in \until{n}$, are identically zero, the
  dynamics~\eqref{eq:CT-Laplacian-optimization-NEW} is linear and has
  $\SS$ as equilibria.  Consider the coordinate transformation from
  $(\xnet,\znet) $ to $(\xnet, \ynet) = (\xnet,\beta \xnet+\znet )$,
  with $ \beta \in \realpositive $ to be chosen later.  The dynamics
  can be rewritten as
  \begin{align}
    \begin{pmatrix}\label{eq:A_linear}
      \dot \xnet
      \\
      \dot \ynet
    \end{pmatrix}
    =
    \Alin
    \begin{pmatrix}
      \xnet
      \\
      \ynet \end{pmatrix}, \quad \mathrm{where} \quad \Alin=
    \begin{pmatrix} 
      -(\alpha-\beta) \Blap & - \Blap\\
      (-\beta(\alpha-\beta)+1)\Blap & -\beta \Blap
    \end{pmatrix}.
  \end{align}
  Consider the candidate Lyapunov function $ \Lyap(\xnet,\ynet) =
  \xnet^T \xnet + \ynet^T\ynet$. Its Lie derivative is the quadratic
  form defined by the matrix
  \begin{align*}
    Q = \identity{2 nd} \Alin + \Alin^T \identity{2nd}=
   \begin{pmatrix}
     -(\alpha-\beta) (\Blap+\Blap^T) & -
     \Blap+(-\beta(\alpha-\beta)+1)\Blap^T
     \\
     (-\beta(\alpha-\beta)+1)\Blap-\Blap^T & -\beta (\Blap+\Blap^T)
   \end{pmatrix}.
 \end{align*}
 Select $\beta$ now satisfying $ \beta^2-\alpha \beta+2=0 $ (this
 equation has a real solution if $\alpha \ge 2 \sqrt{2}$). Then,
 \begin{equation}\label{eq:Q_with_y}
   Q=\begin{pmatrix} 
     -(\frac{\beta^2+2}{\beta}-\beta) & -1\\
     -1 & -\beta
   \end{pmatrix}\otimes(\Blap+\Blap^T) .
 \end{equation}
 Each eigenvalue $ \eta $ of $ Q $ is of the form $ \eta = \lambda
 \frac{-(\beta^2+2)\pm\sqrt{(\beta^2+2)^2-4\beta^2}}{2\beta}$, where $
 \lambda $ is an eigenvalue of $ \Lap+\Lap^T $. Since $\Bgraph$ is
 strongly connected and weight-balanced, $ \Blap+\Blap^T $ is positive
 semidefinite with a simple eigenvalue at zero, and hence $ \eta \leq
 0 $. By the LaSalle invariance principle, the solutions
 of~\eqref{eq:CT-Laplacian-optimization-NEW} from any initial
 condition $ (\xnet_0,\ynet_0) \in \real^{nd} \times \real^{nd} $,
 asymptotically converge to the set $ S=\{(\xnet,\ynet) \ | \
 Q(\xnet,\ynet)^T = \zeros_{2nd}\}\cap W_{\znet_0}$.  To conclude the
 result, we need to show that $ S\subseteq \SS $.  This follows from
 noting that, for $\beta>0$, $ Q(\xnet,\ynet)^T = \zeros_{2nd}$
 implies that $ (\Blap+\Blap^T)\xnet=\zeros_{nd} $ and $
 (\Blap+\Blap^T)\ynet=\zeros_{nd} $, i.e., $ (\xnet,\ynet) \in \SS $.
\end{proof}

The reason behind the introduction of the parameter $\alpha$
in~\eqref{eq:CT-Laplacian-optimization-NEW} comes from the following
observation: if one tries to reproduce the proof of
Theorem~\ref{theorem:dis_opt} for a digraph, one encounters indefinite
terms of the form $(\xnet-\xnet^*)^T(\Blap-\Blap^T)(\znet-\znet^*)$ in
the Lie derivative of $V$, invalidating it as a Lyapunov
function. However, the proof of
Lemma~\ref{lemma:necessary_linear_directed-NEW} shows that an
appropriate choice of $\alpha$, together with a suitable change of
coordinates, makes the quadratic form defined by the identity matrix a
valid Lyapunov function.  We next build on these observations to
establish our main result: the
dynamics~\eqref{eq:CT-Laplacian-optimization-NEW} solves in a
distributed way the optimization problem~\eqref{eq:dis_opt} on
strongly connected weight-balanced digraphs.

\begin{theorem}\longthmtitle{Asymptotic convergence
    of~\eqref{eq:CT-Laplacian-optimization-NEW} on weight-balanced 
    digraphs}\label{theorem:dis_opt-directed}  
  Let $\Bgraph$ be a strongly connected, weight-balanced digraph and
  consider the optimization problem~\eqref{eq:dis_opt}, where each $
  f^i $, $ i \in \{1,\ldots,n\}$, is convex and differentiable with
  globally Lipschitz continuous gradient.  Let $ K \in \realpositive $
  be the Lipschitz constant of $ \gradient \fLift $ and define $
  \map{h}{\realpositive}{\real} $ by
  \begin{align}\label{eq:h}
    h(r) = \frac{1}{2}\smallnonzeroeig (\Lap+\Lap^T)
    \left(-\frac{r^4+3r^2+2}{r}+\sqrt{\left(
          \frac{r^4+3r^2+2}{r}\right)^2-4} \right) +\frac{K
      r^2}{(1+r^2)} ,
  \end{align}
  where $ \smallnonzeroeig(\cdot) $ denotes the non-zero eigenvalue
  with smallest absolute value. Then, there exists $ \beta^* \in
  \realpositive $ with $ h(\beta^*)=0 $ such that, for all $ 0 < \beta
  < \beta^* $, the projection onto the first component of any
  trajectory of~\eqref{eq:CT-Laplacian-optimization-NEW} with $
  \alpha=\frac{\beta^2+2}{\beta} $ asymptotically converges to the set
  of solutions of~\eqref{eq:dis_opt_2}. Moreover, if $ f $ has a
  finite number of critical points, the limit of the projection onto
  the first component of each trajectory is a solution
  of~\eqref{eq:dis_opt_2}.
\end{theorem}
  
\begin{proof}
  For convenience, we denote the
  dynamics~\eqref{eq:CT-Laplacian-optimization-NEW} by $
  \map{\modoptdyn}{\real^{n d} \times \real^{n d}}{\real^{n d} \times
    \real^{n d}} $. Note that the equilibria of $\modoptdyn $ are
  precisely the set of saddle points of $F$ in~\eqref{eq:F_opt}.  Let
  $\xnet^* = \ones_n \otimes x^*$ be a solution
  of~\eqref{eq:dis_opt_2}.  First, note that given any initial
  condition $ (\xnet_0,\znet_0) \in \real^{nd} \times \real^{nd} $,
  the set $ W_{\znet_0} $ defined by~\eqref{eq:S_z0}
  is invariant under the evolutions
  of~\eqref{eq:CT-Laplacian-optimization-NEW}. By
  Proposition~\ref{prop:equiv-F}(i) and~(iii), there exists $
  (\xnet^*,\znet^*) \in \equilibria{\modoptdyn}\cap W_{\znet_0} $.
  Consider the function $ \map{V}{\real^{nd} \times
    \real^{nd}}{\real_{\geq0}} $,
  \[
  \Lyap(\xnet,\znet) = \frac{1}{2}(\xnet-\xnet^*)^T(\xnet-\xnet^*) +
  \frac{1}{2}(\ynet_{(\xnet,\znet)}-\ynet_{(\xnet^*,\znet^*)})^T
  (\ynet_{(\xnet,\znet)}-\ynet_{(\xnet^*,\znet^*)}),
  \]
  where $ \ynet_{(\xnet,\znet)}=\beta\xnet+\znet $ and $ \beta \in
  \realpositive $ satisfies $ \beta^2-\alpha \beta+2=0 $.  This
  function is quadratic, hence smooth.  Next, we consider its Lie
  derivative along $\modoptdyn$ on $W_{\znet_0}$. For $(\xnet,\znet)
  \in W_{\znet_0}$, let
  \begin{align*}
    \xi &= \Lie_{\modoptdyn}V(\xnet,\znet)= (-\alpha \Blap\xnet-\Blap
    \znet -\gradient \fLift(\xnet), \Blap\xnet)^T \nabla
    V(\xnet,\znet)
    \\
    &=\frac{1}{2}\begin{pmatrix} (\xnet-\xnet^*)^T,
      (\ynet_{(\xnet,\znet)}-\ynet_{(\xnet^*,\znet^*)})^T \end{pmatrix}
    \Alin \begin{pmatrix} \xnet,
      \ynet_{(\xnet,\znet)} \end{pmatrix}^T-(\xnet-\xnet^*)^T\gradient
    \fLift(\xnet)
    \\
    & \quad +\frac{1}{2}\begin{pmatrix} \xnet^T,
      \ynet_{(\xnet,\znet)}^T \end{pmatrix} \Alin^T
    \begin{pmatrix}
      \xnet-\xnet^*,     
      \ynet_{(\xnet,\znet)}-\ynet_{(\xnet^*,\znet^*)} 
    \end{pmatrix}^T
    -\beta(\ynet_{(\xnet,\znet)}-\ynet_{(\xnet^*,\znet^*)})^T\gradient
    \fLift(\xnet),
  \end{align*}
  where $ \Alin $ is given by~\eqref{eq:A_linear}. This equation can
  be written as
  \begin{align*}
    \xi =&\frac{1}{2}
    \begin{pmatrix}
      (\xnet-\xnet^*)^T,
      (\ynet_{(\xnet,\znet)}-\ynet_{(\xnet^*,\znet^*)})^T
    \end{pmatrix}
    Q
    \begin{pmatrix}
      \xnet-\xnet^*,
      \ynet_{(\xnet,\znet)}-\ynet_{(\xnet^*,\znet^*)}
    \end{pmatrix}^T - (\xnet-\xnet^*)^T\gradient \fLift(\xnet)
    \\
    &+\begin{pmatrix} (\xnet-\xnet^*)^T,
      (\ynet_{(\xnet,\znet)}-\ynet_{(\xnet^*,\znet^*)})^T \end{pmatrix}
    \Alin \begin{pmatrix} \xnet^*,
      \ynet_{(\xnet^*,\znet^*)} \end{pmatrix}^T-
    \beta(\ynet_{(\xnet,\znet)}-\ynet_{(\xnet^*,\znet^*)})^T\gradient
    \fLift(\xnet),
  \end{align*}
  where $ Q $ is given by~\eqref{eq:Q_with_y}. Note that $ \Alin (
  \xnet^*, \ynet_{(\xnet^*,\znet^*)} )^T = -(
  \Blap\ynet_{(\xnet^*,\znet^*)}, \beta\Blap\ynet_{(\xnet^*,\znet^*)}
  )^T = ( \gradient \fLift(\xnet^*), \beta\gradient
  \fLift(\xnet^*))^T$.  Thus,
  after substituting for $ \ynet_{(\xnet,\znet)} $, we have
  \begin{align}\label{eq:z-issue}
    \xi =&\frac{1}{2}\begin{pmatrix} (\xnet-\xnet^*)^T,      
      (\znet-\znet^*)^T \end{pmatrix}^T \tilde{Q}
    \begin{pmatrix} \xnet-\xnet^*, \znet-\znet^* \end{pmatrix}^T
    \nonumber
    \\
    &-(1+\beta^2)(\xnet-\xnet^*)^T(\gradient \fLift(\xnet)-\gradient
    \fLift(\xnet^*)) -\beta(\znet-\znet^*)^T(\gradient \fLift(\xnet) -
    \gradient \fLift(\xnet^*)),
  \end{align}
  where
  \begin{align*}
    \tilde{Q}=\begin{pmatrix}
      -\beta^3-(\frac{\beta^2+2}{\beta})-\beta & -(1+\beta^2)
      \\
      -(1+\beta^2) & -\beta
      \\
    \end{pmatrix}\otimes (\Blap+\Blap^T).
  \end{align*}
  Each eigenvalue of $ \tilde{Q} $ is of the form
  \begin{align}\label{eq:eigenvalues-tildeQ}
    \tilde{\eta} = \lambda\times \frac{-(\beta^4+3\beta^2+2) \pm
      \sqrt{(\beta^4+3\beta^2+2)^2-4\beta^2}}{2\beta},
  \end{align}
  where $ \lambda $ is an eigenvalue of $ \Lap+\Lap^T $.  Using the
  cocoercivity of $ \fLift $, we can upper bound $\xi$ as,
  \begin{align}\label{eq:one-more}
    \xi \leq \frac{1}{2} \begin{pmatrix} \xnet-\xnet^*
      \\
      \znet-\znet^*\\
      \gradient \fLift(\xnet)-\gradient \fLift(\xnet^*)
      \end{pmatrix}^T
      \underbrace{\begin{pmatrix} \tilde{Q}_{11} & \tilde{Q}_{12} & 0\\
          \tilde{Q}_{21} &\tilde{Q}_{22} & -\beta \identity{nd}\\
          0 & -\beta \identity{nd} & -\tfrac{1}{K} (1+\beta^2) \identity{nd}
        \end{pmatrix}}_{\mathbf{Q}}
      \begin{pmatrix}
        \xnet-\xnet^*
        \\
        \znet-\znet^*
        \\
        \gradient \fLift(\xnet)-\gradient \fLift(\xnet^*)
      \end{pmatrix} ,
  \end{align}
  where $ K \in \realpositive $ is the Lipschitz constant for the
  gradient of $ \fLift $. 

  Since $ (\xnet, \znet) \in W_{\znet_0} $, we have $(\ones_{n}^T
  \otimes \identity{d}) (\znet - \znet^*) =\zeros_{d}$ and hence it is
  enough to establish that~$\mathbf{Q}$ is negative semidefinite on
  the subspace $\WW = \setdef{(v_1,v_2,v_3) \in
    (\real^{nd})^3}{(\ones_{n}^T \otimes \identity{d}) v_2 =
    \zeros_d}$.  Using the fact that $-\tfrac{1}{K} (1+\beta^2)
  \identity{nd} $ is invertible, we can express $\mathbf{Q}$ as
  \begin{align*}
    \mathbf{Q} = N
    \begin{pmatrix}
      \bar{Q} & 0
      \\
      0 & -\tfrac{1}{K} (1+\beta^2) \identity{nd}
    \end{pmatrix}
    N^T
    , \;
    \bar{Q}=\tilde{Q}+\frac{K\beta^2}{(1+\beta^2)}
    \left(\begin{matrix}
        0 & 0\\
        0 & \identity{nd}
      \end{matrix}\right)
    , \;
    N =
    \begin{pmatrix}
      \identity{nd} & 0 & 0
      \\
      0 & \identity{nd} & \frac{\beta K}{1+\beta^2} \identity{nd}
      \\
      0 & 0 & \identity{nd}
    \end{pmatrix}
    .
  \end{align*}
  Noting that $\WW$ is invariant under $N^T$ (i.e., $N^T \WW = \WW$),
  all we need to check is that the matrix $\left(
    \begin{smallmatrix}
      \bar{Q} & 0
      \\
      0 & -\tfrac{1}{K} (1+\beta^2) \identity{nd}
    \end{smallmatrix} \right) $ is negative semidefinite on
  $\WW$. Clearly, $ -\tfrac{1}{K} (1+\beta^2) \identity{nd}$ is
  negative definite. On the other hand, on $(\real^{nd})^2$, $0$ is an
  eigenvalue of $\tilde{Q}$ with multiplicity $2d$ and eigenspace
  generated by vectors of the form $(\ones_n \otimes a,0)$ and
  $(0,\ones_n \otimes b)$, with $a,b \in \real^d$. However, on $
  \setdef{(v_1,v_2) \in (\real^{nd})^2}{(\ones_{n}^T \otimes
    \identity{d}) v_2 = \zeros_d}$, $0$ is an eigenvalue of
  $\tilde{Q}$ with multiplicity $d$ and eigenspace generated by
  vectors of the form $(\ones_n \otimes a,0)$. Moreover, on $
  \setdef{(v_1,v_2) \in (\real^{nd})^2}{(\ones_{n}^T \otimes
    \identity{d}) v_2 = \zeros_d}$, the eigenvalues of
  $\frac{K\beta^2}{(1+\beta^2)} \left(\begin{smallmatrix}
      0 & 0\\
      0 & \identity{nd}
    \end{smallmatrix}\right)
  $ are $\frac{K\beta^2}{(1+\beta^2)}$ with multiplicity $nd-d$ and
  $0$ with multiplicity $nd$. Therefore, using Weyl's
  theorem~\cite[Theorem 4.3.7]{RAH-CRJ:85}, we deduce that the nonzero
  eigenvalues of the sum $ \bar{Q} $ are upper bounded by $
  \smallnonzeroeig(\tilde{Q})+\frac{K\beta^2}{(1+\beta^2)}
  $. From~\eqref{eq:eigenvalues-tildeQ} and the definition of $h$
  in~\eqref{eq:h}, we conclude that the nonzero eigenvalues of~$
  \bar{Q} $ are upper bounded by $ h(\beta) $.  It remains to show
  that there exists $ \beta^* \in \realpositive $ with $ h(\beta^*)=0
  $ such that for all $ 0 < \beta < \beta^* $ we have $ h(\beta) < 0
  $.  For $ r>0 $ small enough, $h(r)<0$, since $ h(r) =
  -\frac{1}{2}\smallnonzeroeig(\Lap+\Lap^T)r+O(r^2)$. Furthermore, $
  \lim_{r\rightarrow \infty} h(r) =K > 0 $. Hence, the existence of $
  \beta^* $ follows from the Mean Value Theorem.  Therefore we
  conclude $ \Lie_{\modoptdyn}V (\xnet,\znet) \leq 0$.  As a
  by-product, the trajectories
  of~\eqref{eq:CT-Laplacian-optimization-NEW} are
  bounded. Consequently, all assumptions of the LaSalle Invariance
  Principle are satisfied and its application yields that any
  trajectory of~\eqref{eq:CT-Laplacian-optimization-NEW} starting from
  an initial condition $ (\xnet_0,\znet_0) $ converges to the largest
  positively invariant set $M$ in $ S_{\modoptdyn,\Lyap} \cap
  W_{\znet_0} $.  Note that if $(\xnet,\znet) \in S_{\modoptdyn,\Lyap}
  \cap W_{\znet_0} $, then $ N^T \left(
    \begin{smallmatrix}
      \xnet-\xnet^*
      \\
      \znet-\znet^*
      \\
      \gradient \fLift(\xnet)-\gradient \fLift(\xnet^*)
    \end{smallmatrix}
  \right) \in \ker (\bar{Q}) \times \{0\} $.  From the discussion
  above, we know $\ker (\bar{Q})$ is generated by vectors of the form
  $(\ones_n \otimes a,0)$, and hence this implies that $\xnet =
  \xnet^* + \ones_n \otimes a$, $\znet = \znet^*$, and $ \gradient
  \fLift(\xnet)=\gradient \fLift(\xnet^*)$, from where we deduce that
  $\xnet$ is also a solution to~\eqref{eq:dis_opt_2}.  Finally, for
  $(\xnet,\znet) \in M$, an argument similar to the one in the proof
  of Theorem~\ref{theorem:dis_opt} establishes $ (\xnet, \znet) \in
  \equilibria{\modoptdyn} $. If the set of equilibria is finite,
  convergence to a point is also guaranteed.
\end{proof}

Figure~\ref{fig:sim} illustrates the result of
Theorem~\ref{theorem:dis_opt-directed} for the network of
Example~\ref{example:counter}.

{ 
\psfrag{0}[][rr]{{}}
\psfrag{1}[rr][rr]{{\tiny $1$}}
\psfrag{3}[rr][rr]{{\tiny $3$}}

\psfrag{-3}[rr][rr]{{\tiny $-3$}}
\psfrag{-5}[rr][rr]{{\tiny $-5$}}

\psfrag{-1}[rr][rr]{{\tiny $-1$}}
\psfrag{5}[rr][rr]{{\tiny $5$}}
\psfrag{15}[rr][rr]{{\tiny $15$}}
\psfrag{25}[rr][rr]{{\tiny $25$}}
\psfrag{35}[rr][rr]{{\tiny $35$}}
\psfrag{40}[rr][rr]{{\tiny $40$}}
\psfrag{45}[rr][rr]{{\tiny $45$}}

\psfrag{10}[rr][rr]{{\tiny $10$}}
\psfrag{20}[rr][rr]{{\tiny $20$}}
\psfrag{30}[rr][rr]{{\tiny $30$}}

\psfrag{-0.5}[rr][rr]{{}}
\psfrag{0.5}[rr][rr]{{}}
\psfrag{1.5}[rr][rr]{{}}
\psfrag{4}[rr][rr]{{}}
\psfrag{2}[rr][rr]{{}}
\psfrag{-2}[rr][rr]{{}}
\psfrag{-4}[rr][rr]{{}}

\begin{figure}[htb]
  \centering
  {\psfrag{2}[rr][rr]{{\tiny $2$}}
  \subfigure[]{\includegraphics[width=.31\linewidth]{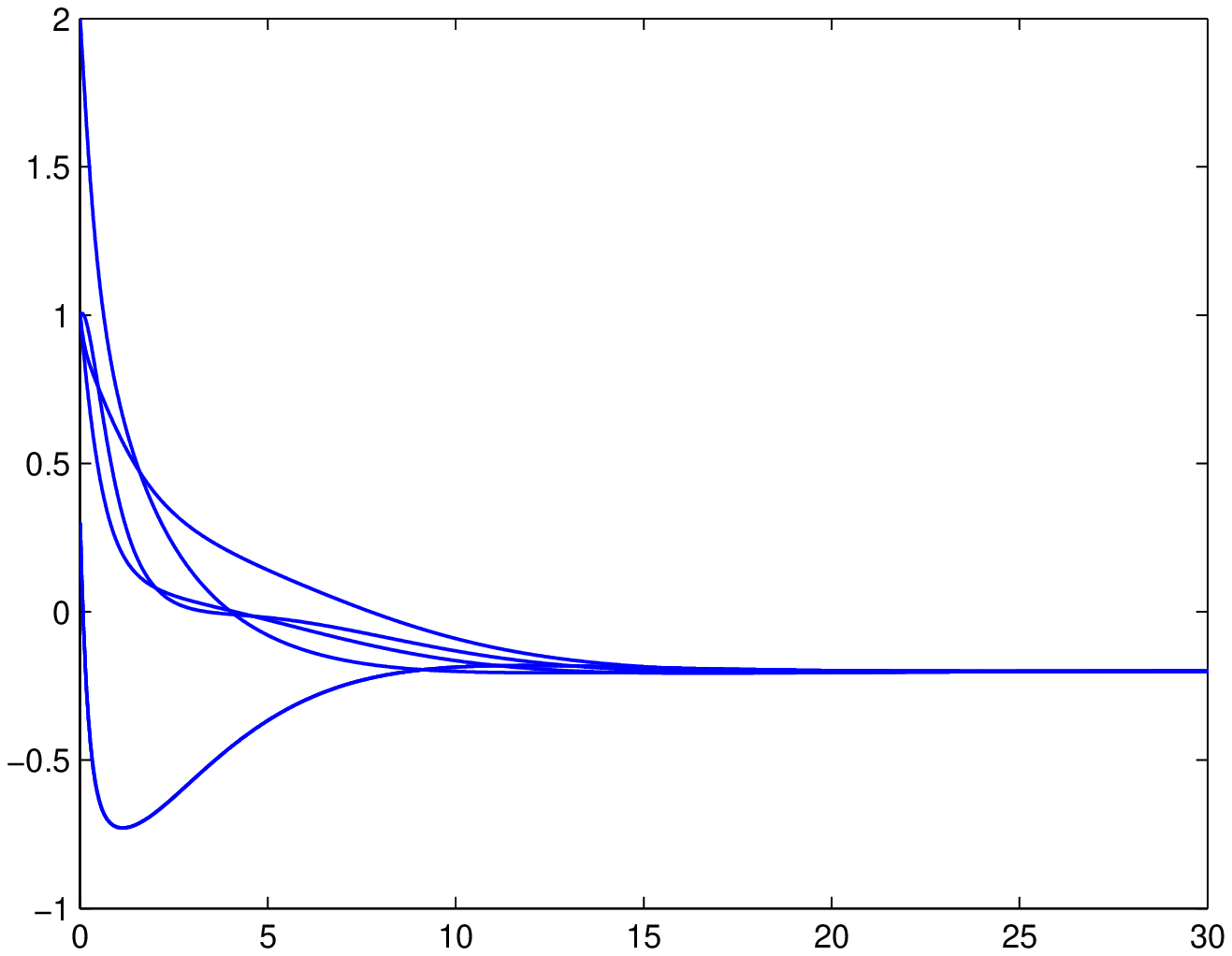}}
  }
  \subfigure[]{\includegraphics[width=.31\linewidth]{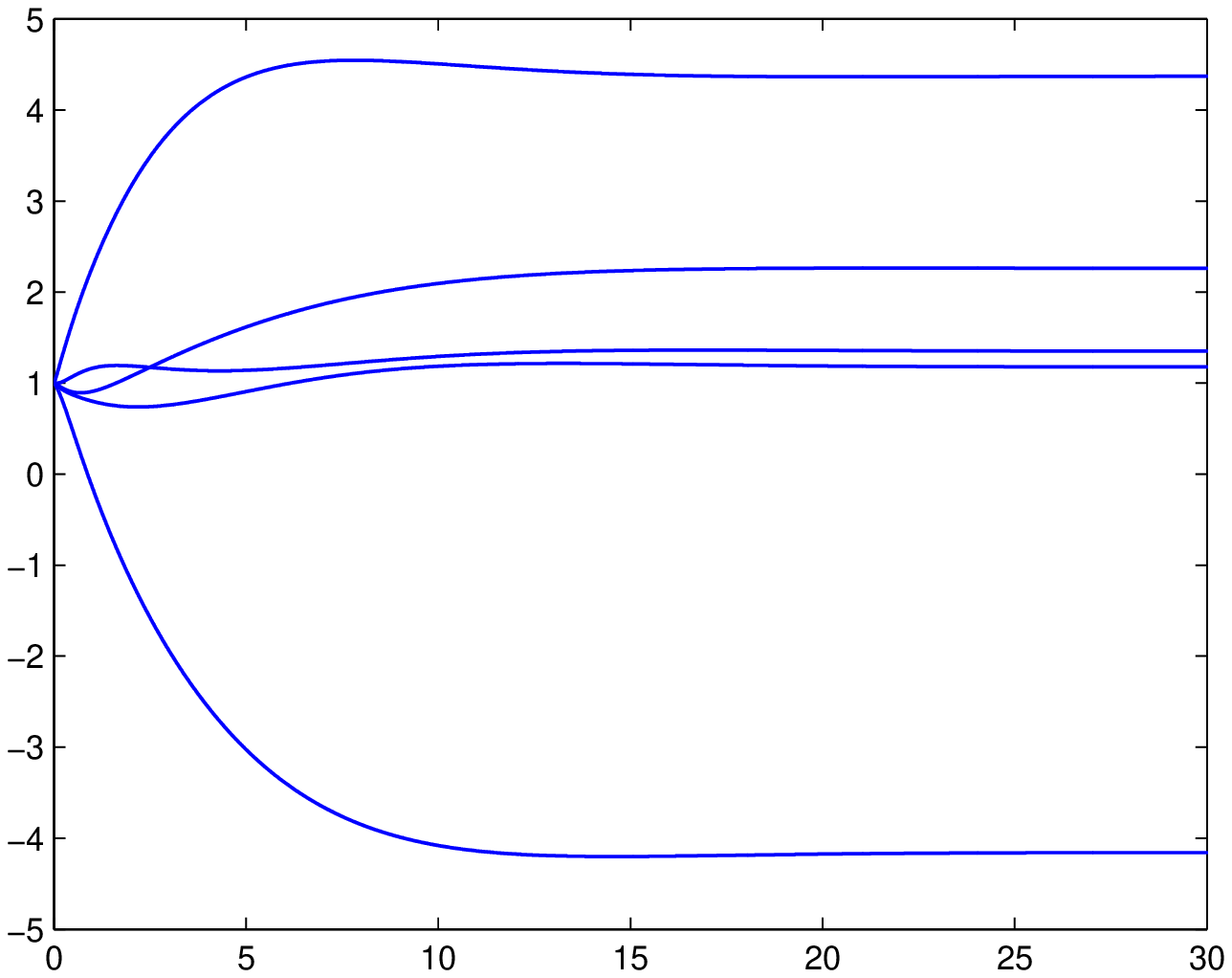}}
  \subfigure[]{\includegraphics[width=.31\linewidth]{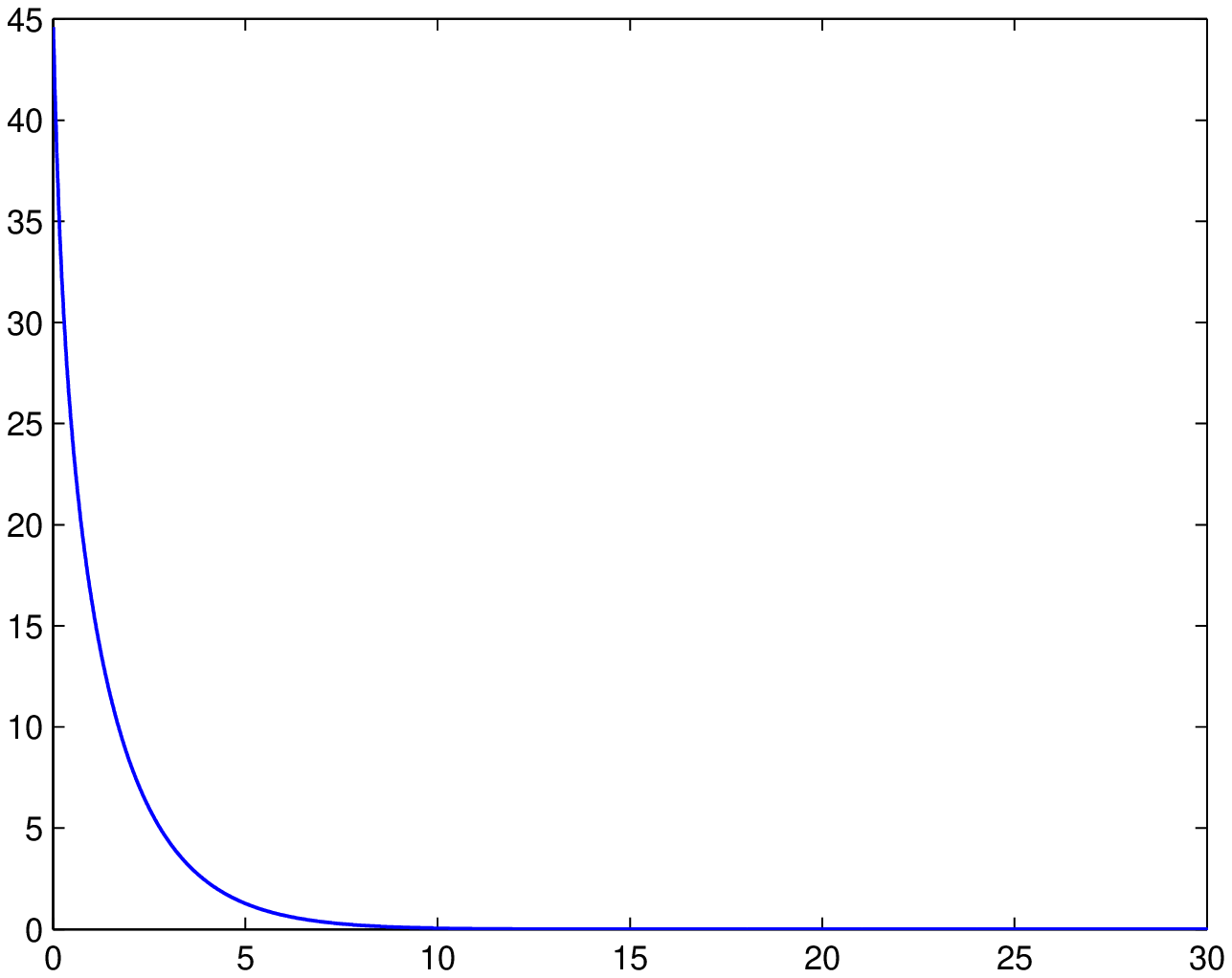}}
  \vspace*{-.5ex}
  \caption{Execution of~\eqref{eq:CT-Laplacian-optimization-NEW} for
    the network of Example~\ref{example:counter} with $ f^1(x)=e^x $,
    $ f^2(x)=(x-3)^2 $, $ f^3(x)=(x+3)^2 $, $ f^4(x)=x^4 $, $ f^5(x)=4
    $. (a) and (b) show the evolution of the agent's values in $ x $
    and $ z $, respectively, and (c) shows the value of the Lyapunov
    function.  Here, $ \alpha =3 $, $ \xnet_0=(1,2,0.3,1,1)^T $, and $
    \znet_0=\ones_5 $.  The equilibrium $ (\xnet^*,\znet^*) $ is $
    \xnet^*=-0.2005\cdot\ones_5 $ and $ \znet^*= (1.1784, 4.3717,
    -4.1598, 2.2598, 1.3499)^T $.}\label{fig:sim}
  \vspace*{-3ex}
\end{figure}
}

\begin{remark}[Locally Lipschitz objective
  functions]\label{remark:diff_assumption}
  Our simulations suggests that the convergence result in
  Theorem~\ref{theorem:dis_opt-directed} holds true for any locally
  Lipschitz objective function.  However, our proof cannot be
  reproduced for this case because it would rely on the generalized
  gradient being globally Lipschitz which, by
  Proposition~\ref{prop:coco_diff}, would imply that the function is
  differentiable.
  \oprocend
\end{remark}

\begin{remark}[Selection of $ \alpha $
  in~\eqref{eq:CT-Laplacian-optimization-NEW}]
  According to Theorem~\ref{theorem:dis_opt-directed}, the parameter
  $\alpha$ is determined by $ \beta $ as $\alpha =
  \frac{\beta^2+2}{\beta} $. In turn, one can observe
  from~\eqref{eq:h} that the range of suitable values for $\beta$
  increases with higher network connectivity and smaller variability
  of the gradient of the objective function.  From a control design
  viewpoint, it is reasonable to choose the value of $\beta$ that
  yields the smallest $\alpha$ while satisfying the conditions of the
  theorem statement.  
  \oprocend
\end{remark}

\begin{remark}[Discrete-time counterpart
  of~\eqref{eq:CT-Laplacian-optimization}
  and~\eqref{eq:CT-Laplacian-optimization-NEW}]
  It is worth noticing that the discretization
  of~\eqref{eq:CT-Laplacian-optimization} for undirected graphs
  (performed in~\cite{JW-NE:10} for the case of continuously
  differentiable, strictly convex functions)
  and~\eqref{eq:CT-Laplacian-optimization-NEW} for weight-balanced
  digraphs gives rise to different discrete-time optimization
  algorithms from the ones considered
  in~\cite{MR-RN:04,AN-AO:09,PW-MDL:09,AN-AO-PAP:10,BJ-MR-MJ:09,MZ-SM:12}.
  \oprocend
\end{remark}

\vspace{-.75ex}
\section{Conclusions and future work}\label{sec:conclusions}

We have studied the distributed optimization of a sum of convex
functions over directed networks using consensus-based dynamics.
Somewhat surprisingly, we have established that the convergence
results established in the literature for undirected networks do not
carry over to the directed scenario. Nevertheless, our analysis has
allowed us to introduce a slight generalization of the saddle-point
dynamics of the undirected case which incorporates a design
parameter. We have proved that, for appropriate parameter choices,
this dynamics solves the distributed optimization problem for
differentiable convex functions with globally Lipschitz gradients on
strongly connected and weight-balanced digraphs. Our technical
approach relies on a careful combination of notions from stability
analysis, algebraic graph theory, and convex analysis.
Future work will focus on the extension of the convergence results to
locally Lipschitz functions in the weight-balanced directed case and
to general digraphs, the incorporation of local and global
constraints, the design of distributed algorithms that allow the
network to agree on an optimal value of the design parameter, the
discretization of the algorithms, and the study of the potential
connections with  dynamic consensus strategies.


\vspace*{-2ex}

\appendix

\section{Appendix}\label{sec:appendix}

The next result shows that the differentiability hypothesis of
Proposition~\ref{prop:coco_nes_suff} cannot be relaxed.

\begin{proposition}[Lipschitz generalized gradient and
  differentiability]\label{prop:coco_diff}
  Any locally Lipschitz function with globally Lipschitz generalized
  gradient is differentiable.
\end{proposition}

\begin{proof}
  Let $ \map{f}{\real^d}{\real} $ be a locally Lipschitz function and
  has a globally Lipschitz generalized gradient map~\cite{
    JC:08-csm-yo}.  Take $ x\in \real^d $ and let us show that
  $ \partial f(x) $ is a singleton.  Since~$f$ is differentiable
  almost everywhere, there exists a sequence of points $
  \{x_n\}_{n=1}^{\infty} $, where $ f $ is differentiable such that $
  \lim_{n \rightarrow \infty} x_n = x $.  Using the set-valued
  Lipschitz property of $ \partial f $, we have $
  \partial f (x) \subset \gradient f (x_n) + K || x_n-x|| B(0,1), $
  where $ K \in \realpositive $ is the Lipschitz constant and $ B(0,1)
  $ is the ball centered at $ 0 \in \real^d $ of radius one.  Hence,
  any element $ v \in \partial f(x) $ can be written as $ v= \gradient
  f (x_n) + K ||x_n-x|| u_n $, where $ u_n $ is a unit vector in $
  \real^d $.  Now, taking the limit, $ v = \lim_{n\rightarrow \infty}
  \gradient f (x_n) $.  Hence the generalized gradient is
  singleton-valued.  Differentiability follows now from the set-valued
  Lipschitz condition.
\end{proof}

\begin{lemma}[Generalized gradient flow from a critical
  point]\label{lemma:flow-in-critical}
  Let $ \map{f}{\real^d}{\real} $ be locally Lipschitz and convex, and
  let $x^*$ be a minimizer of $f$. Then, the only solution of
  $\dot{x}(t)\in -\partial f(x(t))$ starting from $x^* $ is $ x(t)=x^*
  $, for all $ t \geq 0 $.
\end{lemma}
\begin{proof}
  We reason by contradiction. Assume $x(t)$ is not identically
  $x^*$. Since $f$ is monotonically nonincreasing along the gradient
  flow, the trajectory must stay in the set of minimizers of~$f$, and
  hence $t \mapsto f(x(t))$ is constant.  Let $t'$ be the smallest
  time such that $ -\partial f(x^*) \ni v=\dot x(t') \neq
  0$. Using~\cite[Lemma~1]{AB-FC:99}, we have $ 0 = \frac{d}{dt}
  f(x(t)) = v^T \xi$, for all $\xi \in \partial f(x^*)$. In
  particular, for $\xi = -v$, we get $0 = - \TwoNorm{v}^2$, which is a
  contradiction.
\end{proof}

\end{document}